\documentclass[12pt]{amsart}
  \usepackage{latexsym} 
  \usepackage[all]{xy}
  \usepackage{amsfonts} 
  \usepackage{amsthm} 
  \usepackage{amsmath} 
  \usepackage{amssymb}
  \usepackage{pifont}  
  \newcommand{\cC}{{C}}  

  \def\<{{\langle}} 
  \def\>{{\rangle}}

  \def\eps{\varepsilon}

  \def\note#1{{}} 
 \def\can{{\rm \textsf{can}}}

  \def\note#1{}

  \def\cC{{\mathcal C}}

  \def\lhom#1#2#3{{}{\sb{#1}{\rm Hom}(#2,#3)}} 
  \def\rhom#1#2#3{{{\rm Hom}\sb{#1}(#2,#3)}}

  \def\rend#1#2{{{\rm End}\sb{#1}(#2)}}

  \def\Rhom#1#2#3{{{\rm Hom}\sp{#1}(#2,#3)}}

  \def\C{\mathbb{C}}

  \def\beq{\begin{equation}} 
  \def\eeq{\end{equation}}

  \def\id{\mathrm{id}}

  \def\ot{{\otimes}}

    \def\oan#1{\Omega^{#1}\!A}
    \def\oa{\Omega A}
    \def\obn#1{\Omega^{#1}\!B}
    \def\ob{\Omega B}
    \def\Mro{{}^M\!\varrho}

  \newcounter{zlist}

  \newcounter{blist} 
  \newenvironment{blist}{\begin{list}{(\alph{blist})}{ 
  \usecounter{blist}\leftmargin2.5em\labelwidth2em\labelsep0.5em 
  \topsep0.6ex 
  \parsep0.3ex plus0.2ex minus0.1ex}}{\end{list}} 

  \newcounter{rlist}

\def\stac#1{\raise-.2cm\hbox{$\stackrel{\displaystyle\otimes}{\scriptscriptstyle{#1}}$}}

\def\cten#1{\raise-.2cm\hbox{$\stackrel{\displaystyle\widehat{\otimes}}
{\scriptscriptstyle{#1}}$}}

  \def\Label#1{\label{#1}\ifmmode\llap{[#1] }\else 
  \marginpar{\smash{\hbox{\tiny [#1]}}}\fi} 
  \def\Label{\label}

\swapnumbers
  \newtheorem{proposition}{Proposition}[section]
  \newtheorem{lemma}[proposition]{Lemma} 
  \newtheorem{corollary}[proposition]{Corollary}

  \theoremstyle{definition} 
  \newtheorem{definition}[proposition]{Definition}
   \newtheorem{statement}[proposition]{} 
  \newtheorem{example}[proposition]{Example} 

  \theoremstyle{remark}

  \newcounter{c} 
   
  \newcommand{\etyk}[1]{\vspace{-7.4mm}$$\begin{equation}\Label{#1} 
  \addtocounter{c}{1}} 
  \renewcommand{\]}{\ifnum \value{c}=1 $$\else \end{equation}\fi} 
\begin{document} 

 \title{Non-commutative connections of the second kind} 
 \author{Tomasz Brzezi\'nski}
 \address{ Department of Mathematics, Swansea University, 
  Singleton Park, \newline\indent  Swansea SA2 8PP, U.K.} 
  \email{T.Brzezinski@swansea.ac.uk}   
    \date{January 2008} 
  \begin{abstract} 
 A connection-like objects, termed {\em hom-connections} are defined in the realm of non-commutative geometry. The definition is based on the use of  homomorphisms rather than tensor products. It is shown that hom-connections arise naturally from (strong) connections in non-commutative principal bundles. The induction procedure of hom-connections via a map of differential graded algebras or a differentiable bimodule is described. The curvature for a hom-connection is defined, and it is shown that flat hom-connections give rise to a chain complex.  
 \end{abstract} 
  \maketitle

\section{Introduction}
The theory of connections in non-commutative geometry is well-established; see \cite{Con:ncg}. One starts with a differential graded algebra $\oa = \bigoplus_{n=0}\oan n$ over an algebra $A= \oan 0$, and defines a connection in a left $A$-module $M$ as a linear map $\nabla^0: M\to \oan 1\ot_A M$ that satisfies the Leibniz rule $\nabla^0(am) = da \ot_A m + a\nabla^0(m)$, for all $m\in M$ and $a\in A$. This is a non-commutative definition obtained by a direct replacement of commutative algebras (of functions on a manifold $X$) and their modules of sections (of a vector bundle over $X$) in the classical definition of a connection, by non-commutative algebras and their general (one-sided) modules. This definition captures very well the classical context in which connections appear and brings it successfully to the realm of non-commutative geometry. 

On the algebraic side, however, this definition of a connection seems to be only a half of a more general picture. First, a non-commutative connection  is defined with the use of the tensor functor. The tensor functor has a left adjoint, the hom-functor. It seems therefore natural to  ask whether it is possible to introduce  connection-like objects defined with the use of the hom-functor.  Second, the vector space dual to $M$ is a right $A$-module. A (left) connection in the above sense, does not induce  a (right) connection on the dual of $M$. In view of the hom-tensor relation (or, equivalently, the adjointness properties mentioned earlier) the induced map necessarily involves the hom-functor. 

The aim of this paper is to demonstrate that there is a natural and potentially quite rich theory of connection-like objects, defined as maps on the spaces of morphisms of modules. Because of the role played by the spaces of homomorphisms, these objects are termed {\em hom-connections}.

We start by introducing the notion of a hom-connection in Section~2. We then give the classical geometric interpretation of this notion in terms of vector bundles. Next we show that every left connection on a bimodule (in the sense of \cite{CunQui:alg}) gives rise to a hom-connection. This construction then leads to the main non-commutative geometric example: A strong connection in a non-commutative principal bundle or a coalgebra-Galois extension yields a hom-connection on the space of equivariant maps from the total algebra of the extension to a corepresentation space of the structure coalgebra. We then study the induction procedure of hom-connections via differentiable bimodules and, as a particular case, via maps of differential graded algebras. Section 3 is devoted to analysis of hom-connections on higher forms. It is shown that any hom-connection can be extended to higher forms. The notion of a curvature is introduced and we show that a consecutive application of hom-connections can always by expressed in terms of the curvature. This leads to a chain complex associated to a flat (i.e.\ curvature-zero) hom-connection. This chain complex and its homology can be considered as dual complements of the cochain complex associated to a connection and the twisted cohomology, which play a fundamental role in the theory of non-commutative differential fibrations \cite{BegBrz:ser}. The curvature of a hom-connection induced by a map of differential graded algebras is shown to be very closely related to the curvature of the original hom-connection. In particular this induction procedure preserves flatness. Finally we give the interpretation of flat hom-connections with respect to a semi-free differential algebra as contramodules. 

All vector spaces, algebras (always associative and with unit) are over a fixed field $k$. The unadorned tensor product is over $k$. Differential (non-negatively) graded algebras are denoted by $\oa$, $\ob$, the letter $A$, respectively $B$ indicating the zero-degree subalgebra. The $n$-degree subspace of $\oa$ is denoted by $\oan n$. Degree one differentials are denoted by $d$. As is customary, we write $da$ for $d(a)$ with understanding that $d$ is applied only to the element represented by a letter which immediately follows $d$.

\section{ Non-commutative hom-connections}
\begin{definition}\label{def.hom.con}
Given a differential graded algebra $(\oa, d)$ over an algebra $A$, a right {\em hom-connection} is a pair $(M,\nabla_0)$, where $M$ is a right $A$-module and 
$$
\nabla_0 :\rhom A {\oan 1} M \to M,
$$
is a $k$-linear map, such that, for all $f\in \rhom A {\oan 1} M$ and $a\in A$,
$$
\nabla_0(fa) = \nabla_0(f)a + f(da).
$$
Here $\rhom A {\oan 1} M $ is a right $A$-module by $(fa)(\omega) := f(a\omega)$, $\omega \in \oan 1$.
\end{definition}
Since the difference of two maps $\rhom A {\oan 1} M \to M$ satisfying conditions of Definition~\ref{def.hom.con} is a right $A$-module map, for a fixed $M$, hom-connections are an affine space over $\rhom A {\rhom A {\oan 1} M} M$.

\begin{statement}
Take a smooth manifold $X$ and a smooth vector bundle $E$ over $X$. The sections of any vector bundle over $X$ are a (right) module over the algebra of smooth functions $C^\infty(X)$. By the Serre-Swan theorem, there is an isomorphism
$$
\rhom {C^\infty(X)} {\Gamma(T^*X)} {\Gamma (E)} \simeq \mathrm{Vect}_X (T^*X, E),
$$
where $\mathrm{Vect}_X(T^*X, E)$ denotes the space of vector bundle maps $T^*X\to E$, and $\Gamma(-)$ denotes (smooth) sections. $\mathrm{Vect}_X (T^*X, E)$ is a (right) $C^\infty(X)$-module with the fibrewise product ($\mathrm{Vect}_X (T^*X, E)$ can be identified with the module of sections on the hom-bundle $\mathrm{Hom}(T^*X, E)$ over $X$).  With this identification, a hom-connection $(\Gamma(E), \nabla_0)$ corresponds to a map
$$
\nabla_0 : \mathrm{Vect}_X (T^*X, E) \to \Gamma(E),
$$
such that, for all $\varphi \in \mathrm{Vect}_X (T^*X, E)$ and $f\in C^\infty(X)$,
$$
\nabla_0(\varphi f)= \nabla_0(\varphi) f +\varphi\circ df.
$$
\end{statement}

\begin{statement} 
The difference between hom-connections and connections can be exemplified in the case when $\oan 1$ is a finitely generated and projective right $A$-module. In this case, there are isomorphisms, for any left $A$-module $N$,
$$
\rhom k N {\oan 1\ot_A N} \simeq \rhom k N {\lhom A {\oan 1^*} N} \simeq \lhom A {\oan 1^*} {\rend k N},
$$
where $\rend k N$ is a left $A$-module by $(ae)(n) := ae(n)$, for all $a\in A$, $n\in N$ and $e\in \rend k N$. Also, $\oan 1^* := \rhom A {\oan 1} A$ is viewed as an $A$-bimodule by $a\xi b (\omega) = a\xi(b\omega)$. Under this identification, a connection (in the sense of \cite{Con:ncg}) $\nabla^0: N \to \oan 1\ot_A N$, corresponds to a collection of $k$-linear maps $\{\nabla^0_\xi: N\to N \; |\; \xi \in \oan 1^*\}$, such that, for all $a\in A$ and $n\in N$,
$$
\nabla^0_{a\xi} (n) = a\nabla^0_\xi (n), \qquad \nabla^0_\xi (an) = \nabla^0_{\xi a} (n) +\xi(da) n.
$$

On the other hand, for any right $A$-module $M$,
$$
\rhom k {\rhom A {\oan 1} M} M \simeq \rhom k{M\ot_A \oan 1^*} M \simeq 
\lhom A {\oan 1^*} {\rend k M},
$$
where $\rend k M$ is a left $A$-module by $(ae)(m) := e(ma)$, for all $a\in A$, $m\in M$ and $e\in \rend k M$. A hom-connection $(M,\nabla_0)$ corresponds to a collection of $k$-linear maps $\{\nabla_0^\xi: M\to M \; |\; \xi \in \oan 1^*\}$, 
such that, for all $a\in A$ and $m\in M$,
$$
\nabla_0^{a\xi} (m) = \nabla_0^\xi (ma), \qquad \nabla_0^{\xi a} (m) = \nabla_0^\xi (m)a +m\xi(da) .
$$
\medskip
\end{statement}

\begin{statement}
In case both $A$ and $\oan 1$ are finite-dimensional, there is a one-to-one correspondence between hom-connections and connections in a comodule of the dual coalgebra $A^*$ (cf.\  \cite[Definition~6]{Brz:fla}).

Asumme that $A$ is a finite-dimensional algebra and denote by $C$ its dual coalgebra, $C= \rhom k A k$. Let $\{a^s\in A\}$ be a basis for $A$ and let $\{c^s\in C\}$ be the dual basis. Write $L = \rhom k {\oan 1} k$ for the dual vector space of $\oan 1$. $L$ is an $A$-bimodule, with multiplications given by  $(alb)(\omega) = l(b\omega a)$, for all $l\in L$, $a,b\in A$ and $\omega \in \oan 1$. Since $A$ is finite-dimensional, $L$ is a $C$-bicomodule with the left and right coactions
$$
l\mapsto \sum_s c^s \ot l a^s, \qquad l\mapsto \sum_s a^sl\ot c^s.
$$
 The differential $d: A\to \oan 1$ induces a coderivation on $L$ (cf.\ \cite{Doi:hom}) 
$$
\lambda: L\to C, \qquad l\mapsto l\circ d\ ,
$$
i.e.\ $\lambda = \rhom k d k$. Let $M$ be a right $A$-module. Then $M$ is a left $C$-comodule with the coaction $\Mro: m\mapsto \sum_s c^s \ot ma^s$. If, in addition, $\oan 1$ is finite dimensional, then there is an isomorphism of vector spaces
$$
\Upsilon: L\Box_C M \to \rhom A {\oan 1} M, \qquad \Upsilon\left(\sum_i l_i \ot m_i\right)(\omega) = \sum_i l_i(\omega)m_i,
$$
where 
$$
L\Box_C M = \left\{ \sum_i l_i \ot m_i \in L\ot M \; |\; \sum_{i,s} a^sl_i \ot c^s \ot m_i = \sum_{i,s} l_i \ot c^s\ot m_ia^s\right\}
$$ 
is the cotensor product. 
\medskip

\noindent {\bf Proposition.} {\em  Given a finite dimensional algebra $A$ and a finite-dimensional $\oan 1 $, define $C= \rhom k A k$, $L = \rhom k {\oan 1} k$ and $\lambda = \rhom k dk$. The assignment
$$
\nabla_0 \mapsto - \nabla_0 \circ \Upsilon,
$$
establishes a bijective correspondence between hom-connections $\nabla_0 : \rhom A {\oan 1} M\to M$  and connections $L\Box_C M  \to M$ in a left $C$-comodule $M$ (with respect to the coderivation $\lambda : L\to C$).}
\medskip 

\begin{proof}
First note that, for all $\sum_i l_i \ot m_i \in L\Box_C M$ and  $a\in A$,
\begin{equation}\label{triangle}
\Upsilon\left(\sum_i l_i \ot m_i\right) a = \Upsilon\left(\sum_i l_i a \ot m_i\right).
\end{equation}
Take a hom-connection $\nabla_0$, write $\bar{\nabla} : =  - \nabla_0 \circ \Upsilon$, and compute
\begin{eqnarray*}
\left(\Mro \circ \bar{\nabla}\right)\!\!\!\!\!\!\!\!\!\!\!\!\!\!\!&& \left(\sum_i l_i \ot m_i\right) = - \sum_s c^s \ot \nabla_0\left( \Upsilon \left(\sum_i l_i \ot m_i\right)\right) a^s \\
&=& - \sum_s c^s \ot \nabla_0\left( \Upsilon \left(\sum_i l_i \ot m_i\right) a^s \right) +
\sum_s c^s \ot  \Upsilon \left(\sum_i l_i \ot m_i\right)\left(d a^s \right)\\
&=&  - \sum_s c^s \ot \nabla_0\left( \Upsilon \left(\sum_i l_i a^s\ot m_i\right)  \right) +
\sum_{i,s} c^s \ot   l_i (da^s) m_i\\
&=&  \sum_s c^s \ot \bar{\nabla}\left( \sum_i l_i a^s\ot m_i\right) +
 \sum_i \lambda(l_i)\ot m_i,
\end{eqnarray*}
where the second equality follows by the fact that $\nabla_0$ is a hom-connection, while the third one follows by equation \eqref{triangle}. In the view of the definition of the left $C$-coaction on $L$, this calculation confirms that $\bar{\nabla}$ is a connection in the left $C$-comodule $(M,\Mro)$.

Let $\{\omega^t\}$ be a finite basis of the vector space $\oan 1$ and let $\{e^t\in L\}$ be its dual basis. In terms of this basis the inverse of $\Upsilon$ is $\Upsilon^{-1}: f\mapsto \sum_t e^t\ot f(\omega^t)$.  Note that, for all $a\in A$, $\omega \in \oan 1$, $f\in \rhom A {\oan 1} M$,
$$
 f(a\omega) = \sum_t e^t(a\omega)f(\omega^t),
$$
i.e.,
\begin{equation}\label{a-eq}
fa =  \sum_t e^ta \ot f(\omega^t).
\end{equation}
Take a connection $\bar{\nabla} : L\Box_C M\to M$ in the left $C$-comodule $(\Mro, M)$, and set $\nabla_0 = \bar{\nabla}\circ \Upsilon^{-1}$. Since $\bar{\nabla}$ is a connection in a comodule, for all $\sum_i l_i \ot m_i \in L\Box_C M$,
$$
\sum_s c^s \ot \bar{\nabla} \left(\sum_i l_i \ot m_i\right) a^s =  \sum_s c^s \ot \bar{\nabla}\left( \sum_i l_i a^s\ot m_i\right) +
 \sum_i \lambda(l_i)\ot m_i.
 $$
Evaluating this equality at $a\in A$, one immediately finds
\begin{equation}\label{a-eq2}
\bar{\nabla} \left(\sum_i l_i \ot m_i\right) a = \bar{\nabla} \left(\sum_i l_i a \ot m_i\right)  + \sum_i l_i(da)m_i.
\end{equation}
Combining equations \eqref{a-eq} and \eqref{a-eq2} one obtains
\begin{eqnarray*}
\nabla_0(fa) &=& -\bar{\nabla}\left(\sum_t e^ta \ot f(\omega^t)\right) = -\bar{\nabla}\left(\sum_t e^t \ot f(\omega^t)\right) a + \sum_t e^t(da) f(\omega^t)\\
&=& \nabla_0(f)a + f(da).
\end{eqnarray*}
This proves that $(M,\nabla_0)$ is a hom-connection.
\end{proof}

\end{statement}

\begin{example}\label{ex.inner} Suppose that $(\oan 1, d)$ is an {\em inner} first order differential calculus on $A$, that is that there is a one-form $\Xi\in \oan 1$ such that, for all $a\in A$, 
$$
da = a\Xi -\Xi a.
$$
Then, for any right $A$-module $M$, 
$$
\nabla^\Xi_0: \rhom A {\oan 1} M \to M, \qquad f\mapsto f(\Xi),
$$
is a hom-connection. 

The universal first order differential calculus on $A$, i.e.\ the differential graded algebra with $\oan 1$ given as the kernel of the multiplication map $\mu: A\ot A\to A$ and with $d: a\mapsto 1\ot a - a\ot 1$, is inner if and only if $A$ is a separable algebra. This means  that there exists an element $\iota \in A\ot A$ such that, for all $a\in A$,   $a\iota=\iota a$ and $\mu(\iota) =1$; see e.g.\ \cite[Chapter II\S 1]{DeMIng:sep}. The generating universal one-form is defined by $\Xi = \iota-1\ot 1$. Thus every module over a separable algebra admits a hom-connection with respect to the universal differential graded algebra.
\end{example}

\begin{example} As a very explicit example of a hom-connection, we calculate all hom-connections on the Laurent polynomial algebra $A=\C [u,u^{-1}]$,  with respect to the non-commutative differential calculus $\oan 1 := du A$ with the $A$-bimodule structure determined by $udu =q du u$, $q\in \C\setminus \{0\}$.  If $q\neq 1$, the differential calculus $\oan 1$ is inner with the generating form 
$$
\Xi = \frac{1}{q-1}duu^{-1}.
$$
Write $f(u)= \sum_n c_n u^n$ for a general element of $A$, and then, for any $\gamma\in \C$,  denote by  $f(\gamma u)$ the  polynomial $\sum_n c_n\gamma^n u^n$. Take any right $A$-module $M$. Then $M\simeq \rhom A {\oan 1} M$, where the element $m\in M$ is sent to $R_m \in \rhom A {\oan 1} M$ defined by $R_m(du f(u)) = mf(u)$. Note that $f(u)du = du f(qu)$, hence $R_m f(u) = R_{mf(qu)}$. Furthermore
$$
df(u) = du \left(\partial_qf(u)\right), \qquad \partial_q f(u) = \frac{f(qu)-f(u)}{q-1} u^{-1},
$$ 
(in case $q=1$ the Jackson $q$-derivative $\partial_q$ should be replaced by the usual derivative). With these computations at hand, one easily checks that hom-connections $\rhom A {\oan 1} M \to M$ are in bijective correspondence with $\C$-linear maps $\nabla_0: M\to M$ such that, for all $m\in M$, $f(u)\in A$,
$$
\nabla_0(mf(qu)) = \nabla_0 (m) f(u) + m \partial_q f(u).
$$
The hom-connection $\nabla_0^\Xi$ comes out as
$$
\nabla_0^\Xi(m) = \frac{1}{q-1}mu^{-1}, \qquad (q\neq 1).
$$
In particular, in the case of a regular $A$-module $M=A$ hom-connections are in one-to-one correspondence with elements $a$ of $A$, by
$$
\nabla_0^a(f(u)) = a f(q^{-1}u) + \partial_q f(q^{-1}u).
$$
Note that $a = \nabla_0^a(1)$. The hom-connection $\nabla_0^\Xi$ corresponds to $a =  \frac{1}{q-1}u^{-1}$.
\end{example}

\begin{statement}\label{diff.bim}  Consider differential graded algebras $\oa$ and $\ob$. There is a procedure of induction of hom-connections with respect to $\ob$ from hom-connections with respect to $\oa$  by a {\em differentiable bimodule}, a notion introduced in \cite[Definition~2.10]{BegBrz:ser} as an extension of ideas developed in 
 \cite[Section~3.6]{Mad:int}, \cite{DubMad:cur}.
 
A 
$(B,A)$-bimodule $E$ with a left connection 
$\nabla^0: E\to \obn 1\ot_{B} E$ and a $(B,A)$-bimodule map 
$\sigma: E\ot_{A}\oan 1\to \obn 1\ot_{B}E$ is called a 
{\em differentiable $(\ob,\oa)$-bimodule}, provided that, for all $e\in E$ and
$a\in A$, 
\begin{equation} \label{eq.a}
\nabla^0(ea) = \nabla^0(e)a +\sigma(e\ot_{A}da). 
\end{equation}
For all $e\in E$, we write $\sigma_e$ for a right $A$-linear map $\oan 1\to \obn 1\ot_B E$, $\omega\mapsto \sigma(e\ot_A \omega)$.\medskip

\noindent {\bf Theorem.} {\em  Let $(E,\nabla^0,\sigma)$ be a differentiable $(\ob,\oa)$-bimodule. For any hom-connection $\nabla_0: \rhom A {\oan 1} M\to M$, there is a hom-connection $(\rhom A E M, \nabla_0^E)$, 
$$
\nabla^E_0:  \rhom B {\obn 1} {\rhom A E M} \to \rhom A E M,
$$
defined by
$$
\nabla^E_0 (f)(e) := \nabla_0(f\circ \sigma_e) - f\left(\nabla^0\left( e\right)\right),
$$
for all $f\in \rhom B {\obn 1} {\rhom A E M}  \simeq \rhom A {\obn 1\ot_B E} M$ and $e\in E$.}\medskip

\begin{proof}
First note that, since $\sigma$ is defined on the tensor product of $A$-modules, for all $a\in A$, $(f\circ\sigma_e)a = f\circ\sigma_{ea}$. Hence
\begin{eqnarray*}
\nabla_0^E(f)(ea) &=& \nabla_0\left(\left(f\circ \sigma_e\right) a\right) - f\left(\nabla^0\left( ea\right)\right)\\
&=& \nabla_0\left(f\circ \sigma_e\right) a + (f\circ\sigma_e)(da) -  f\left(\nabla^0(e)a\right) -f\left(\sigma(e\ot_{A}da)\right) \\
&=& \nabla_0\left(f\circ \sigma_e\right) a  - f\left(\nabla^0(e)\right)a = \nabla_0^E(f)(e)a,
\end{eqnarray*}
where the second equality follows by the definition of a hom-connection and by equality \eqref{eq.a} which is a part of the definition of a differentiable bimodule. The third equality is a consequence of the right $A$-linearity of $f$. This proves that $\nabla_0^E(f)$ is a right $A$-module map, as required.

To check the hom-connection property of $\nabla_0^E$, in addition to any $f$ and $e$, take also any $b\in B$, observe that the left $B$-linearity of $\sigma$ implies that $fb\circ\sigma_e = f\circ\sigma_{be}$, and compute
\begin{eqnarray*}
\nabla_0^E(fb)(e) &=& \nabla_0(fb\circ \sigma_e) - fb\left(\nabla^0\left( e\right)\right)
= \nabla_0(f\circ\sigma_{be}) - f\left(b\nabla^0\left( e\right)\right)\\
&=& \nabla_0(f\circ\sigma_{be}) - f\left(\nabla^0\left(b e\right)\right) +f(db\ot_B e)
= \nabla_0^E(f)(be) +f(db)(e),
\end{eqnarray*}
where the penultimate equality follows by the Leibniz rule for the left connection $\nabla^0$. Therefore, $\nabla_0^E$ is a hom-connection.
\end{proof}

\end{statement}

\begin{statement}\label{con.hom-con}
A left connection in a bimodule  is a special case of the differentiable bimodule discussed in \ref{diff.bim}. Recall  from \cite[Section~8]{CunQui:alg} that
given an algebra $B$, a left connection on an $(A,B)$-bimodule $M$ is a connection $\nabla^0: M\to \Omega^1A\ot_A M$ that is a right $B$-module map. As a consequence of Theorem~\ref{diff.bim} one obtains
\medskip

\noindent {\bf Corollary.} {\em  Let $\nabla^0: M\to \oan 1\ot_A M$ be a left connection in an $(A,B)$-bimodule $M$. Then, for any right $B$-module $N$, the pair $(\rhom B M N,\nabla_0)$, where
$$
\nabla_0 := - \rhom B {\nabla^0} N: \rhom B {\oan 1 \ot_A M} N \to \rhom B M N,
$$
is a hom-connection.}
\medskip 

\begin{proof}
Since $\nabla^0$ is a right $B$-module map, $(M,\nabla^0, 0)$ is a differentiable $(\oa, \ob)$-bimodule, with $\ob =B$. The map $\nabla_0$ is induced from the zero hom-connection in $M$ by the procedure described in Theorem~\ref{diff.bim}.  
\end{proof}
\end{statement}

\begin{example}
The left $A$-module $A$ has a left connection $d: A\mapsto \oan 1\ot_A A\simeq \oan 1$. Let $B$ be any subalgebra of the constant algebra $H^0(A)=\{a\in A \; |\; da =0\}$. Then $d$ is a left connection in an $(A,B)$-bimodule $A$, hence, for any right $B$-module $N$, it induces a  hom-connection in $\rhom B A N$, $\nabla_0(f) (a) = -f(da)$. Thus, although $A$ might not have a hom-connection, its vector space dual always has a hom-connection.
\end{example}

\begin{statement} \label{strong} An example of a hom-connection of more geometric origin comes from the theory of (strong) connections on non-commutative principal bundles; see e.g.\ \cite{Haj:str}, \cite{BrzMaj:geo}, \cite{DabGro:str}.

Let $C$ be a coalgebra and let $P$ be an algebra and a right $C$-comodule with the coaction $\Delta_P: P\to P\ot C$. Define the coinvariant subalgebra of $P$,
$$
A:= \{ a\in P \: | \; \forall p\in P, \; \Delta_P(ap) = a\Delta_P (p) \},
$$
and consider the {\em canonical} map
$$
\can : P\ot_A P\to P\ot C, \qquad p\ot_A q\mapsto p\Delta_P(q).
$$
The algebra inclusion $A\subseteq P$ is called a {\em coalgebra-Galois extension}, provided the canonical map is bijective. A coalgebra-Galois extension is a rudimentary version of a non-commutative principal bundle; the bijectivity of $\can$ represents freeness of the action of a structure group, represented by $C$, on the total space of a principal bundle, represented by $P$.

A {\em strong connection} in a  coalgebra-Galois extension is a $C$-covariant splitting of the module of universal one-forms on $P$ into horizontal and vertical parts. It is described equivalently as a $k$-linear map $D: P\to \left(\oan 1\right) P \simeq \oan 1 \ot_AP$ such that
\begin{blist}
\item for all $a\in A$ and $p\in P$, $D(ap) = aD(p) + da\ot_A p$,
\item $(\id \ot \Delta_P)\circ D = (D\ot \id)\circ \Delta_P$.
\end{blist}
Here $\oan 1$ denotes the universal differential structure on $A$ (the kernel of the multiplication map). Since $P$ is a right $C$-comodule, it is a right module for the opposite of the convolution algebra of $C$, $B=(C^*)^{op}$. By the construction of $A$, the coaction $\Delta_P$ is left $A$-linear, hence $P$ is an $(A,B)$-bimodule. The condition (a) means that $D$ is a left connection in $P$, while condition (b) means that $D$ is right $B$-linear. In a word: $D$ is a left bimodule connection in $P$.

Take a right $C$-comodule $V$. The $C$-coaction canonically induces the right $B$-module structure on $V$. Furthermore 
$$
\rhom B P V \simeq \Rhom C P V,
$$
where the latter denotes the vector space of $C$-colinear maps $P\to V$. $\Rhom C P V$ is formally dual to a noncommutative vector bundle with standard fibre $V$ associated to $A\subseteq P$ (such a bundle is defined as $\Rhom C V P$). By \ref{con.hom-con}, a strong connection $D$ gives rise to a hom-connection $(\Rhom C P V , \nabla_0=-\Rhom C D V)$.
\end{statement}

\begin{statement}\label{ind} A map of differential graded algebras $\theta: \oa \to \ob$ allows one to view both $B$ and $\obn 1$  as $A$-bimodules by (the zero degree component of) $\theta$. For any right $A$-module $M$, there are also natural isomorphisms of vector spaces,
$$
\rhom B {\obn 1} {\rhom A B M} \simeq \rhom A {\obn 1\ot_B B} M \simeq \rhom A {\obn 1} M.
$$
Furthermore, $B$ is a differentiable $(\ob,\oa)$-bimodule with the connection $b\mapsto db$ and with the twist map $\sigma: B\ot_A \oan 1 \to \obn 1\ot_B B\simeq \obn 1$, $b\ot_A \omega \mapsto b\theta(\omega)$; see \cite[Example~2.11]{BegBrz:ser}. With this in mind one can state the following consequence of Theorem~\ref{diff.bim}.\medskip

\noindent {\bf Corollary.} {\em  Given a map $\theta : \oa \to \ob$ of differential graded algebras and a hom-connection $\nabla_0: \rhom A {\oan 1} M\to M$, there is a hom-connection $(\rhom A B M, \nabla_0^\theta)$, 
$$
\nabla^\theta_0:  \rhom B {\obn 1} {\rhom A B M} \to \rhom A B M,
$$
defined by
$$
\nabla^\theta_0 (f)(b) := \nabla_0(f\circ\ell_b\circ \theta) - f(db),
$$
for all $f\in \rhom B {\obn 1} {\rhom A B M}  \simeq \rhom A {\obn 1} M$ and $b\in B$. Here $\ell_b$ denotes the left multiplication by $b\in B$, $\ell_b: \obn 1\to \obn 1$, $\omega\mapsto b\omega$.}\medskip

\begin{proof}
One needs only to observe that, for all $b\in B$, $\sigma_b = \ell_b\circ\theta$, and then apply Theorem~\ref{diff.bim} to the differentiable bimodule $B$.
\end{proof}
\end{statement}

\section{Curvature and flat hom-connections}
\setcounter{equation}{0}

\begin{statement}\label{higher} Any hom-connection $(M,\nabla_0)$ can be extended to higher forms by the Leibniz rule as follows. The vector space $\bigoplus_{n=0} \rhom A {\oan n} M$ is a right module of $\oa$ with the multiplication, for all $\omega\in \oan n$, $f\in \rhom A{\oan{n+m}} M$, $\omega'\in \oan m$,
$$
f\omega (\omega') := f(\omega\omega').
$$
Note that $f\omega \in \rhom A {\oan m} M$. 

For all $n$, define
$$
\nabla_n: \rhom A {\oan {n+1}}M \to \rhom A {\oan {n}}M,
$$
by
$$
\nabla_n(f)(\omega) :=  \nabla_0 (f\omega) + (-1)^{n+1} f(d\omega),
$$
for all $f\in \rhom A {\oan{n+1}}M$ and $\omega \in \oan{n}$.
Note that for $n=0$ this formula reduces to the Leibniz rule for $\nabla_0$ in Definition~\ref{def.hom.con}, once $\rhom AAM$ is canonically identified with $M$.
\end{statement}
\begin{lemma}\label{lemma.Leibniz}
 For all $n\geq 0$, $\omega\in \oan m$ and $f\in \rhom A {\oan{m+n+1}}M $,
$$
\nabla_n(f\omega) = \nabla_{m+n}(f) \omega +(-1)^{m+n} f{d\omega}.
$$
\end{lemma}
\begin{proof}
For $n=0$ this is simply the definition of $\nabla_n$ in \ref{higher}. For $n>0$, take any $\omega' \in \oan n $, $f\in \rhom A {\oan{m+n+1}}M$ and $\omega\in \oan m$. Then, on one hand, the definition of $\nabla_n$ yields
$$
\nabla_n (f\omega)(\omega') = \nabla_0(f\omega\omega') + (-1)^{n+1} (f\omega)(d\omega') = \nabla_0(f\omega\omega') + (-1)^{n+1} f(\omega d\omega').
$$
On the other hand, the definition of $\nabla_{m+n}$ and the (graded) Leibniz rule for $d$ imply that
\begin{eqnarray*}
(\nabla_{m+n}(f) \omega) (\omega') &=& \nabla_{m+n}(f)(\omega\omega')\\
& =& \nabla_0(f{\omega\omega'}) + (-1)^{n+1} f(\omega d\omega') +(-1)^{m+n+1} f(d\omega \omega').
\end{eqnarray*}
Hence,
$
\nabla_n(f\omega) = \nabla_{m+n}(f) \omega +(-1)^{m+n} f{d\omega},
$
as claimed.
\end{proof}

\begin{proposition}\label{prop.curvature}
Let $(M,\nabla_0)$ be a hom-connection. For all $n>0$:
\begin{blist}
\item  $\nabla_{n-1}\circ \nabla_n$ is a right $A$-linear map.
\item Set $F = \nabla_0\circ\nabla_1$, and define
$$
\Theta_n: \rhom A {\oan{n+1}} M \to \rhom A {\oan{n-1}} {\rhom A {\oan 2} M}, $$
$$
\Theta_n(f)(\omega_1)(\omega_2) = f(\omega_1\omega_2).
$$
Then the following diagram
$$
\xymatrix{\rhom A {\oan{n+1}} M \ar[rr]^{\nabla_{n-1}\circ\nabla_n} \ar[rd]_{\Theta_n} && \rhom A {\oan{n-1}} M \\
& \rhom A {\oan{n-1}} {\rhom A {\oan 2} M\ar[ur]_{\;\;\;\;\;\;\; \rhom A {\oan{n-1}} F} } &}
$$
commutes.
\end{blist}
\end{proposition}
\begin{proof}
(a) The application of Lemma~\ref{lemma.Leibniz} yields, for all $f\in \rhom A {\oan{n+1}} M$ and $a\in A$,
\begin{eqnarray*}
\nabla_{n-1}\circ\nabla_n (fa) &=& \nabla_{n-1} (\nabla_n(f) a) +(-1)^n\nabla_{n-1} (f{da}) \\
&=& \nabla_{n-1}(\nabla_n(f))a +(-1)^{n-1}\nabla_n(f){da} + (-1)^n\nabla_{n-1} (f{da}).
\end{eqnarray*}
Take any $\omega \in \oan n$ and, using the (graded) Leibniz rule for $d$ and the definition of the $\nabla_n$,  compute
\begin{eqnarray*}
&&\hspace{-.5in}\left(\nabla_n(f){da} - \nabla_{n-1}(f{da})\right) (\omega) = \nabla_n (f)(da\omega) - \nabla_0(f{da}\omega) - (-1)^n f(dad\omega)\\
&=&  \nabla_0(f{da\omega}) +(-1)^{n+1} f(d(da\omega)) - \nabla_0(f{da\omega}) - (-1)^n f(dad\omega) = 0.
\end{eqnarray*}
Therefore,  $\nabla_{n-1}\circ \nabla_n (fa) =\left(\nabla_{n-1}\circ \nabla_n (f)\right) a$, as required. \bigskip

(b) For all $\omega\in \oan {n-1}$ and $f\in \rhom A {\oan{n+1}} M$,
\begin{eqnarray*}
(\nabla_{n-1}\circ\nabla_n) (f)(\omega) &=& \nabla_0(\nabla_n(f)\omega) + (-1)^n \nabla_n(f)(d\omega) \\
&=& \nabla_0(\nabla_n(f)\omega) + (-1)^n\nabla_0(f{d\omega}).
\end{eqnarray*}
By Lemma~\ref{lemma.Leibniz}, 
$$
\nabla_n(f)\omega = \nabla_1 (f\omega) + (-1)^{n-1} f{d\omega}.
$$
Hence 
$$
(\nabla_{n-1}\circ\nabla_n) (f)(\omega) = \nabla_0 \left(\nabla_1 \left(f\omega\right)\right) = \left(\rhom A {\oan{n-1}} F \circ \Theta_n\right)(f)(\omega),
$$
as claimed. 
\end{proof}

\begin{definition}
The right $A$-module map $F = \nabla_0\circ \nabla_1 : \rhom A {\oan 2} M \to M$ is called a {\em curvature} of a hom-connection $(M,\nabla_0)$. A hom-connection $(M,\nabla_0)$ is said to be {\em flat} if its curvature is the zero map. 
\end{definition}

\begin{corollary}
\begin{blist}
\item Any flat hom-connection $(M,\nabla_0)$ gives rise to the chain complex
$$
\xymatrix{
\cdots \ar[r]^-{\nabla_2}  & \rhom A {\oan 2} M \ar[r]^-{\nabla_1}  & \rhom A {\oan 1} M \ar[r]^-{\nabla_0}  & M.}
$$
The homology of this complex is denoted by $H_*(A; M,\nabla_*)$.
\item The action of $\oa$ on $\bigoplus_{n=0}\rhom A {\oan n} M$ described in \ref{higher} descends to the right action of the cohomology of $\oa$, $H^*(A)$, on  $H_*(A; M,\nabla_*)$
\end{blist}
\end{corollary}
\begin{proof}
Part (a) follows immediately by the commutative diagram in Proposition~\ref{prop.curvature}~(b).
Assertion
(b) is a consequence of Lemma~\ref{lemma.Leibniz}.
\end{proof}

\begin{example}
Suppose that $(\oan 1, d)$ is an {inner} first order differential calculus on $A$ with the generating form $\Xi$, and take the associated hom-connection $(M,\nabla^\Xi_0)$ as in Example~\ref{ex.inner}. Then, for all $f\in  \rhom A {\oan 2} M$ and $\omega\in \oan 1$,
$$
\nabla_1^\Xi(f)(\omega) = f(d\omega) +\nabla_0^\Xi (f\omega) = f(d\omega +\omega\Xi),
$$
hence the curvature of $\nabla_0^\Xi$ comes out as
$$
F(f) = f(d\Xi +\Xi^2).
$$
Thus $(M,\nabla_0^\Xi)$ is a flat connection, provided $d\Xi +\Xi^2 =0$. Conversely, if $M$ cogenerates $\oan 2$ as a right $A$-module, then flatness of $(M,\nabla_0^\Xi)$ implies that $d\Xi +\Xi^2 =0$.
\end{example}

\begin{statement} In the setup of \ref{con.hom-con}, a left connection $\nabla^0: M\to \oan 1\ot_A M$ on an $(A,B)$-bimodule is extended to higher forms by the graded Leibniz rule,
$$
\nabla^n : \oan n \ot_A M\to \oan{n+1} \ot_A M, \qquad \omega \ot_A m \mapsto d\omega \ot_A m + (-1)^n \omega \nabla^0(m).
$$
The higher forms $\nabla_n$ of 
a hom-connection $(\rhom B MN, \nabla_0 = - \rhom B {\nabla^0} N)$ are related to the $\nabla^n$ by
$$
\nabla_n = (-1)^{n+1} \rhom B {\nabla ^n} N.
$$
Thus the curvature of $\nabla_0$ comes out as
$$
F = -\rhom B {\bar{F}} N,
$$
where $\bar{F}$ is the curvature of $\nabla^0$. Consequently,  the hom-connection $(\rhom B MN, \nabla_0 )$ is flat provided $\nabla^0$ is a flat connection. If $N$ is an injective $B$-module, then the homology $H_*(A; \rhom B MN,\nabla_*)$ can be computed from the cohomology of the complex $H^*(A;  M,\nabla^*)$ associated  to the flat connection $\nabla^n$, by applying the functor $\rhom B - N$. 

A hom-connection associated to a flat strong connection in a coalgebra-Galois extension in \ref{strong} is flat.
\end{statement}

\begin{statement} In view of \ref{ind}, a map of differential graded algebras can be used to induce  hom-connections. The curvature of the induced hom-connection turns out to be closely related to the curvature of the original hom-connection.\medskip

\noindent {\bf Proposition.} {\em  Let $\theta : \oa \to \ob$ be a map of differential graded algebras, and  $(\rhom A B M, \nabla_0^\theta)$ be a hom-connection (with respect to $\ob$ ) induced from a hom-connection $\nabla_0: \rhom A {\oan 1} M\to M$ as in \ref{ind}. Denote by $F$ the curvature of $\nabla_0$, and by $F^\theta$ the curvature of $\nabla_0^\theta$. Then, for all $f\in \rhom B {\obn 2} {\rhom A BM} \simeq \rhom A {\obn 2} M$, and all $b\in B$,  
$$
F^\theta(f)(b) = F(fb\circ\theta).
$$
In particular, if $\nabla_0$ is flat, then so is $\nabla_0^\theta$.}\medskip

\begin{proof}
Note that, in view of the definition of the action of the differential graded algebra $\ob$ on the space $\bigoplus_{n=0} \rhom B {\obn n} M$ (see \ref{higher}), the definition of $\nabla^\theta_0$ in Corollary~\ref{ind} can be equivalently written as $
\nabla^\theta_0 (f)(b) = \nabla_0(fb\circ \theta) - f(db)
$. The higher hom-connections $\nabla^\theta_n$ associated to $\nabla_0^\theta$ are defined as in \ref{higher}. This definition of $\nabla_1^\theta$ together with the definition of $\nabla_0^\theta$ and with the Leibniz rule for $d$ imply, for all $b\in B$, $\omega\in \obn 1$ and $f\in \rhom B {\obn 2} {\rhom A BM} \simeq \rhom A {\obn 2} M$,
\begin{eqnarray*}
\nabla^\theta_1(f)(\omega)(b) &=& \nabla^\theta_0 (f\omega) (b) + f(d\omega)(b) \\
&=& \nabla_0(f\omega b\circ \theta) - (f\omega)(db) + f(d\omega b) \\
&=& \nabla_0(f{\omega b}\circ \theta) + f\left( d\left(\omega b\right)\right).
\end{eqnarray*}
Identifying $\rhom B {\obn 2} {\rhom A BM}$ with $\rhom A {\obn 2} M$, one thus obtains
\begin{equation}\label{eq.*}
\nabla^\theta_1 (f)(\omega) = \nabla_0(f\omega\circ\theta) + f(d\omega). 
\end{equation}
Starting with the just derived equality  \eqref{eq.*} and then using the fact that $\theta$ is a map of differential graded algebras and the definition of $\nabla_1$ in terms of $\nabla_0$, one computes, for all $\omega\in \oan 1$ and $f\in \rhom A {\obn 2} M$,
$$
\nabla_1^\theta(f)(\theta(\omega)) = \nabla_0(f{\theta(\omega)}\circ\theta) + f(d\theta(\omega)) 
= \nabla_0 \left(\left(f\circ\theta\right)\omega\right) + (f\circ\theta)(d\omega) = \nabla_1(f\circ\theta)(\omega).
$$
Therefore,
\begin{equation}\label{eq.**}
\nabla_1^\theta(f)\circ \theta = \nabla_1(f\circ \theta). 
\end{equation}
With these formulae at hand one computes the curvature of $\nabla^\theta_0$ as follows:
\begin{eqnarray*}
F^\theta(f)(b) &=& (\nabla_0^\theta\circ\nabla_1^\theta) (f)(b) = \nabla_0\left( \nabla_1^\theta\left( f\right) b\circ\theta\right) - \nabla_1^\theta(f)(db)\\
&=& \nabla_0 \left( \nabla_1^\theta(f) b\circ \theta\right) - \nabla_0(f{db}\circ\theta) \\
&=& \nabla_0\left(\nabla_1^\theta(fb) \circ \theta\right) + \nabla_0(f{db}\circ\theta)  - \nabla_0(f{db}\circ\theta) \\
&=& \nabla_0\left(\nabla_1\left( fb\circ\theta\right)\right) = F(fb\circ \theta).
\end{eqnarray*}
The second equality follows by the definition of $\nabla_0^\theta$ in \ref{ind}, the third equality is a consequence of  \eqref{eq.*}. Since $\nabla^\theta_1$ arises from the hom-connection $\nabla^\theta_0$, it satisfies the Leibniz rule in Lemma~\ref{lemma.Leibniz}. This yields the fourth equality. The penultimate equality follows by  \eqref{eq.**}. This proves the first assertion of the proposition. The second assertion is immediate.
\end{proof}
\end{statement} 

\begin{statement} A differential graded algebra $\oa$ is said to be {\em semi-free} if $\oan {n+1} = \oan n\ot_A \oan 1$, for all $n\geq 1$. As revealed in \cite{Roi:mat}, there is a bijective correspondence between semi-free differential graded algebras over $A$ and $A$-corings with a group-like element. Starting with an $A$-coring $\cC$ \cite{Swe:pre} with coproduct $\Delta: \cC\to \cC\ot_A \cC$, counit $\eps: \cC\to A$ and a group-like element $x\in \cC$, the associated differential graded algebra $\oa_{\cC,x}$ is determined from $\oan 1_{\cC,x} := \ker \eps$, with the differential
$$
da = x\ot_A a -a\ot_A x, \qquad dc = x\ot_A c -\Delta(c) + c\ot_A x,
$$
for all $a\in A$, $c\in \ker\eps$.

Let $j:\ker\eps\to \cC$ be the inclusion map. Take a right $A$-module $M$. Since $\cC \simeq A \oplus \ker\eps$ (cf.\ \cite[28.14]{BrzWis:cor}), the formula, for any $f\in \rhom A \cC M$,
$$
 \varphi (f) = \nabla_0 (f\circ j) + f(x),
$$
establishes a bijective correspondence between $k$-linear maps $ \varphi: \rhom A\cC M\to M$ and $\nabla_0: \rhom A {\oan 1_{\cC,x}} M\to M$. Under this correspondence, $(M,\nabla_0)$ is a hom-connection if and only if $ \varphi$ is a right $A$-module map rendering commutative the following diagram 
$$
\xymatrix{\rhom A  A M \ar[rr]^-{\rhom A  \eps M} \ar[rd]_\simeq && \rhom A  \cC M\ar[dl]^{ \varphi} \\
& M . & }
$$
 The hom-connection $(M,\nabla_0)$ is flat if and only if also the following diagram 
 $$
\xymatrix{
\rhom A  \cC {\rhom A  \cC M} \ar[rrrr]^-{\rhom A  \cC { \varphi}}\ar[d]_\simeq &&&& \rhom A \cC M \ar[d]^{ \varphi} \\
\rhom A {\cC\ot_A\cC} M \ar[rr]^-{\rhom A  \Delta M} && \rhom A  \cC M \ar[rr]^-{ \varphi} && M ,}
$$
commutes. These two conditions mean that $(M, \varphi)$ is a {\em contramodule} for $\cC$; see \cite[Section~III.5]{EilMoo:fou}, \cite{Pos:hom}. Therefore, flat hom-connections with respect to a semi-free differential graded algebra are in bijective correspondence with {contramodules} of the associated coring.

The first order calculus $\oa^1_{\cC,x}$ is inner if and only if $\cC$ is cosplit coring, i.e.\ if and only if there exists an element $\iota \in \cC$ such that, for all $a\in A$, $a\iota=\iota a$ and $\eps(\iota) =1$. The generating one-form is $\Xi = \iota-x$. The  hom-connection defined in Example~\ref{ex.inner} is flat provided $\iota$ is a group-like element.
\end{statement}

  \section*{Acknowledgements} 
  I would like to thank Edwin Beggs, Gabriella B\"ohm, Piotr Hajac and Tomasz Maszczyk for discussions and comments. This paper was completed during author's stay at the Mathematical Institute of the Polish Academy of Sciences (IMPAN) in Warsaw. The support of grants MKTD-CT-2004-509794 and N201 1770 33 is acknowledged.

\end{document}